\documentclass[10pt, onecolumn]{IEEEtran}
\usepackage{epsf}
\usepackage{bm}
\usepackage{latexsym}
\usepackage{amsmath}
\usepackage{amssymb}
\usepackage{amsthm}
\usepackage{graphicx}
\usepackage{color}
\usepackage{cite}
\usepackage{mathtools}
\usepackage{setspace}
\RequirePackage{amsopn}
\RequirePackage{amsfonts}

\mathtoolsset{showonlyrefs}

\title{\LARGE \bf On Robustness Properties in Empirical Centroid Fictitious Play}
\author{BRIAN SWENSON$^{\dagger*}$, SOUMMYA KAR$^\dagger$ AND JO\~{A}O XAVIER$^\star$\thanks{The work was partially supported by the FCT project FCT [UID/EEA/50009/2013] through the Carnegie Mellon/Portugal Program managed by ICTI from FCT and by FCT Grant CMU-PT/SIA/0026/2009, and was partially supported by NSF grant ECCS-1306128. \newline$^\dagger$Department of Electrical and Computer Engineering, Carnegie Mellon University, Pittsburgh, PA 15213, USA (brianswe@andrew.cmu.edu and soummyak@andrew.cmu.edu).\newline $^\star$Institute for Systems and Robotics (ISR), Instituto Superior Tecnico (IST), Technical University of Lisbon, Portugal (jxavier@isr.ist.utl.pt).}}

\newtheorem{theorem}{Theorem}
\newtheorem{lemma}{Lemma}
\newtheorem{prop}{Proposition}

\newtheorem{assumption}{A.}
\newtheorem{definition}{Definition}

\begin{document}

\maketitle
\thispagestyle{empty}
\begin{abstract}
Empirical Centroid Fictitious Play (ECFP) is a generalization of the well-known Fictitious Play (FP) algorithm designed for implementation in large-scale games. In ECFP, the set of players is subdivided into equivalence classes with players in the same class possessing similar properties. Players choose a next-stage action by tracking and responding to aggregate statistics related to each equivalence class. This setup alleviates the difficult task of tracking and responding to the statistical behavior of every individual player, as is the case in traditional FP.
Aside from ECFP, many useful modifications have been proposed to classical FP, e.g., rules allowing for network-based implementation, increased computational efficiency, and stronger forms of learning. Such modifications tend to be of great practical value; however, their effectiveness relies heavily on two fundamental properties of FP: robustness to alterations in the empirical distribution step size process, and robustness to best-response perturbations. The main contribution of the paper is to show that similar robustness properties also hold for the ECFP algorithm. This result serves as a first step in enabling practical modifications to ECFP, similar to those already developed for FP.
\end{abstract}
%\begin{keywords}
%\end{keywords}

\section{Introduction}
The field of learning in games is concerned with the study of systems of interacting agents, and in particular, the question of how simple behavior rules applied at the level of individual agents can lead to desirable global behavior. Fictitious Play (FP) \cite{Brown51} is one of the best studied game-theoretic learning algorithms. While attractive for its intuitive simplicity and proven convergence results, certain practical issues make FP prohibitively difficult to implement in games with a large number of players\cite{marden06,Lambert01,Swenson-MFP-Asilomar-2012,swenson2012ECFP}.

Empirical Centroid FP (ECFP) \cite{Swenson-MFP-Asilomar-2012,swenson2012ECFP} is a recently proposed generalization of FP designed for implementation in large games. In ECFP, the set of players is subdivided into sets of ``equivalence classes'' of players sharing similar properties.
%\footnote{In the extreme case, each player may be grouped into a separate ``equivalence class'' containing only the one player---a scenario equivalent to classical FP. In this way, ECFP generalizes FP.}
In this formulation, players only track and respond to an aggregate statistic (the empirical centroid) for each class of players, rather than tracking and responding to statistical properties of every individual player, as in classical FP.
ECFP has been shown to learn elements of the set of symmetric Nash equilibria for the class of multi-player games known as potential games.

The main focus of this paper will be to study ECFP and show that certain desirable properties possessed by classical FP also hold for the more general ECFP. In particular, the work \cite{leslie2006generalised} studied classical FP and proved that the fundamental learning properties of FP can be retained in the following scenarios:\\
(i) The step size sequence of the empirical distribution process takes on a form other than $\{1/t\}_{t\geq1}$.\\
(ii) Players are permitted to make suboptimal choices when choosing  a next-stage action so long as the degree of suboptimality decays asymptotically to zero with time.
%(ii) Players utility functions are (deterministically or randomly) perturbed with the magnitude of the perturbations decaying asymptotically to zero over time

We say a FP-type algorithm is \emph{step-size robust} if it retains its fundamental learning properties in the first scenario, and we say an algorithm is \emph{best-response robust} if it retains its fundamental learning properties in the second scenario.

The notion of step-size robustness generalizes the concept of the empirical distribution of classical FP. A player's empirical distribution in classical FP is taken to be the time-averaged histogram of the player's action history; implicitly, this has an incremental step size of $1/t$. Scenario (i) allows players to choose alternate step-size sequences. Of particular interest is that it allows for construction of an empirical distribution that places more emphasis on recent observations while discounting observations from the distant past.

The notion of best-response robustness generalizes FP by relaxing the traditional assumption that players are always perfect optimizers. In particular, in classical FP, it is assumed that players are capable of choosing their next-stage action as a (precise) best response to the empirical action history of opposing players.
%In practice, this tends to be a stringent assumption
In practice, this is a stringent assumption, requiring that players have perfect knowledge of the empirical distribution of all opposing players at all times, and are capable of precisely solving a (non-trivial) optimization problem each iteration of the algorithm. By relaxing this implicit assumption slightly (as in scenario (ii)), one is able to consider many useful extensions of FP of both practical and theoretical value.
%We say an algorithm is \emph{robust} or possess the \emph{robustness property} if it retains its fundamental convergence properties under scenario (i).

In \cite{leslie2006generalised}, the best-response robustness of FP was used to show convergence to the set of Nash equilibria of stochastic FP with vanishing smoothing, and to prove convergence of an FP-inspired actor-critic learning algorithm.
In \cite{Lambert01}, best-response robustness of FP was used to show convergence of sampled FP---a variant of FP in which computational complexity is mitigated by approximating the expected utility using a Monte-Carlo method---and used again in \cite{swenson2015CESFP} to ensure convergence of an even more computationally efficient version of sampled FP. In \cite{swenson2015StrongFP}, the best-response robustness of FP is used to construct a variant of FP achieving a strong form of learning in which the player's period-by-period strategies are guaranteed to converge to equilibrium (rather than only convergence in terms of the empirical frequencies, as is typical in FP).
The best-response robustness of FP is also useful in that it allows for practical network-based implementations of FP; e.g., \cite{swenson2012ECFP}.

%The robustness property of FP is also useful in that it allows for practical network-based (i.e., distributed) implementations of FP. In [cite], an implementation of FP is presented in which players are endowed with a preassigned communication graph.\footnote{Comment that ECFP partial `robusness property' for ECFP was presented here. But is of limited value, and would like to have the full robustness property.} Players observe only their own actions and must disseminate any other information via the communication structure using a consensus-type algorithm. The robustness property of FP is indispensable in the proof of this result.

The main contribution of this paper is to demonstrate that ECFP is both step-size robust and best-response robust; i.e., ECFP retains its fundamental learning properties under scenarios (i) and (ii) above.  This result is a necessary first step in order to develop practical modifications for ECFP similar in spirit to those already developed for FP; e.g., improved computational efficiency, network-based implementation rules, and strongly convergent variants of the algorithm, as mentioned above.\footnote{The results of this paper are directly applied in \cite{swenson2015StrongFP} to prove a strong learning result for a variant of ECFP. We also note that one possible network-based implementation of ECFP has been presented in \cite{swenson2012ECFP}. This implementation---which considers a fixed communication graph topologies and synchronous communication rules---relies on a weak form of best-response robustness (see \cite{swenson2012ECFP}, \textbf{A.3}). In order to consider ECFP in more general distributed scenarios (e.g., random communication graph topology and asynchronous communication rules) it is necessary to have the full robustness property derived in this paper.}
We prove the result following a similar line of reasoning to \cite{leslie2006generalised,benaim2005stochastic}; we first study a continuous-time version of ECFP, and then use results from the theory of stochastic approximations to prove our main result regarding convergence of discrete-time ECFP based on properties of the continuous-time counterpart.

%It is to be noted that, prior to the current work, a weaker version robustness property has been known to hold for ECFP; in particular, convergence of ECFP was shown in the case that the error discussed in (i) decays $O(log(t)/t)$. This property was used study an implementation of the algorithm in a network-based setting in which information was disseminated using a synchronous consensus-type algorithm. However, only by considering the full robustness property (convergence of the algorithm is guaranteed regardless of the error decay rate) can further network-based extensions of ECFP be considered (e.g., using asynchronous gossip-type algorithms.)  Moreover, showing the RP for ECFP allows for extensions of ECFP similar to those noted above. In particular, extensions to actor-critic type algorithms [cite], computationally efficient approaches [cite], and strongly convergent approaches [cite], and asynchronous network-based settings [cite].

The remainder of the paper is organized as follows. Section \ref{sec_prelims} sets up the notation to be used in the subsequent development and reviews the classical FP algorithm. Section \ref{sec_DT_ECFP} presents discrete-time ECFP and states the main result. Section \ref{sec_CTS} reviews relevant results in differential inclusions and stochastic approximations to be used in the proof of the main result. Section \ref{sec_CT_ECFP} presents continuous-time ECFP. Section \ref{sec_DT_convergence} proves convergence of discrete-time ECFP using properties of continuous-time ECFP. Section \ref{sec_conclusion} provides concluding remarks.

\section{Preliminaries}
\label{sec_prelims}
\subsection{Game Theoretic Preliminaries}
A review of game-theoretic learning algorithms---including classical FP---can be found in \cite{marden-shamma2013,young2004strategic}.

A normal form game is given by the triple $\Gamma = (N,(Y_i)_{i\in N},(u_i(\cdot))_{i\in N})$, where $N = \{1,\ldots,n\}$ represents the set of players, $Y_i$---a finite set of cardinality $m_i$---denotes the action space of player $i$ and $u_i(\cdot):\prod_{i=1}^n Y_i \rightarrow \mathbb{R}$ represents the utility function of player $i$.

Throughout this paper we assume:
\begin{assumption}
\label{a_iden_interests}
All players use identical utility functions.
%That is, there exists a function $u$ such that for all $i\in n$ and $y\in Y$ $u_i(y) = u(y)$.
\end{assumption}
Under this assumption we drop the subscript $i$ and denote by $u(\cdot)$ the utility function used by all players. The set of mixed strategies for player $i$ is given by $\Delta_i = \{p \in \mathbb{R}^{m_i}:\sum_{k=1}^{m_i} p(k) = 1,~p(k) \geq 0~\forall k=1,\ldots,m_i \}$, the $m_i$-simplex.   A mixed strategy $p_i \in \Delta_i$ may be thought of as a probability distribution from which player $i$ samples to choose an action. The set of joint mixed strategies is given by $\Delta^n = \prod_{i=1}^n \Delta_i$. A joint mixed strategy is represented by the $n$-tuple $\left(p_1,\ldots,p_n \right)$, where $p_i\in \Delta_i$ represents the marginal strategy of player $i$, and it is implicity assumed that players' strategies are independent.

%A pure strategy is a degenerate mixed strategy which places probability $1$ on a single action in $Y_i$. We denote the set of pure strategies by $A_i = \{e_1,e_2,\ldots e_{m_i}  \}$ where $m_i$ is the number of actions available to player $i$, and $e_j$ is the $j$th canonical vector in $\mathbb{R}^{m_i}$. The set of joint pure strategies is given by $A^n = \prod_{i=1}^n A_i$.

The mixed utility function is given by $U(\cdot):\Delta^n \rightarrow \mathbb{R}$, where,
\begin{equation}
U(p_1,\ldots,p_n) := \sum\limits_{y \in Y} u_i(y)p_1(y_1)\ldots p_n(y_n).
\label{mixed_U}
\end{equation}
Note that $U(\cdot)$ may be interpreted as the expected value of $u(y)$ given that the players' mixed strategies are statistically independent. For convenience, the notation $U(p)$ will often be written as $U(p_i,p_{-i})$, where $p_i \in \Delta_i$ is the mixed strategy for player $i$, and $p_{-i}$ indicates the joint mixed strategy for all players other than $i$.
%This paper will often deal with games with identical utility functions such that $U_i(p) = U_j(p)\; \forall i,j$; in such cases we drop the subscript on the utility of player $i$ and write $U(p) = U_i(p)\;\forall i$.= \prod_{i=1}^n A_i$ be the set of joint pure strategies.

For $\epsilon \geq 0$, $i\in N$ and $p_{-i} \in \Delta_{-i}$, define the $\epsilon$-best response set for player $i$ as
$$BR_i^\epsilon(p_{-i}) := \{p_i \in \Delta_i: U(p_i,p_{-i}) \geq \max_{\alpha_i \in \Delta_i} U(\alpha_i,p_{-i}) - \epsilon\}$$
and for $p\in \Delta$ define
\begin{equation}
\label{def_BR}
BR^\epsilon(p) := (BR_1^\epsilon(p_{-1}),\ldots,BR_n^\epsilon(p_{-n})).
\end{equation}
The set of Nash equilibria is given by
\begin{equation}
NE := \{p \in \Delta^n:~ U(p_i,p_{-i}) \geq U(p_i',p_{-i}), \forall p_i' \in \Delta_i,~\forall i\}.
\end{equation}

As a matter of convention, all equalities and inequalities involving random objects are to be interpreted almost surely (a.s.) with respect to the underlying probability measure, unless otherwise stated.

\subsection{Repeated Play}
\label{sec_rep_play}
The learning algorithms considered in this paper assume the following format of repeated play.

Let a normal form game $\Gamma$ be fixed. Let players repeatedly face off in the game $\Gamma$, and for $t\in\{1,2,\ldots\}$, let $a_i(t)\in\Delta_i$ denote the action played by player $i$ in round $t$.\footnote{An action is usually assumed to be pure strategy, or a vertex of the simplex $\Delta_i$. In this work, an action is permitted to be an arbitrary mixed strategy (cf. \cite{leslie2006generalised}, for the case of FP).  Since the results hold for any actions of this form, they also hold for the typical case where actions are restricted to be pure strategies.} Let the $n$-tuple $a(t) = (a_1(t),\ldots,a_n(t))$ denote the joint action at time $t$.

%Let the empirical history distribution (or empirical distribution) of player $i$ be given by\footnote{Note that each $a_i(t)$ is a delta function, and thus the empirical distribution $q_i(t)$ is a normalized histogram of the action choices of player $i$.} $q_i(t) := \frac{1}{t}\sum_{s=1}^t a_i(s)$, and let the joint empirical distribution be given by the $n$-tuple
Denote by $q_i(t) \in \Delta_i$, the \emph{empirical distribution}\footnote{The term \emph{empirical distribution} is often used to refer explicitly to the time-averaged histogram of the action choices of some player $i$; i.e., $q_i(t) = \frac{1}{t}\sum_{s=1}^t a_i(s)$. However, using a broader definition as considered here, allows for interesting algorithmic generalizations; e.g., learning processes that place greater emphasis on observations of more recent actions. See \cite{leslie2006generalised} for further discussion.} of player $i$. The precise manner in which the empirical distribution is formed will depend on the algorithm at hand. In general, $q_i(t)$ is formed as a function of the action history $\{a_i(s)\}_{s=1}^t$ and serves as a compact representation of the action history of player $i$ up to and including the round $t$. The joint empirical distribution is given by $q(t) := (q_1(t),\ldots,q_n(t))$.

\subsection{Classical Fictitious Play}
\label{sec_FP}
FP may be intuitively described as follows. Players repeatedly face off in a stage game $\Gamma$. In any given stage of the game, players choose a next-stage action by assuming (perhaps incorrectly) that opponents are using stationary and independent strategies. In particular, let the empirical distribution be given by the time-averaged histogram
\begin{equation}
\label{def_qt_FP}
q_i(t):= \frac{1}{t} \sum_{s=1}^t a_i(s);
\end{equation}
in FP, players use the empirical distribution of each opponent's past play as a prediction of the opponent's behavior in the upcoming round and choose a next-round strategy that is optimal (i.e., a best response) given this prediction.

A sequence of actions $\{a(t)\}_{t\geq 1}$ such that\footnote{In all learning algorithms discussed in this paper, the initial action $a_i(1)$ may be chosen arbitrarily for all $i$.\label{footnote_initial_cond}}
\begin{equation}
a_i(t+1) \in BR_i(q_{-i}(t)),~\forall i,
\label{FP_BR}
\end{equation}
\noindent for all $t\geq 1$, is referred to as a \emph{fictitious play process}. It has been shown that FP achieves Nash equilibrium learning in the sense that $d(q(t),NE) \rightarrow 0$ as $t\rightarrow \infty$ for select classes of games including two-player zero-sum games \cite{robinson1951iterative}, two-player two-move games \cite{miyasawa1961convergence}, and multi-player potential games \cite{Mond01},\cite{Mond96}.

\subsection{Empirical Centroid FP Setup}
A presentation of ECFP in it's most elementary form (i.e., all players are grouped into a single equivalence class) is given in \cite{swenson2012ECFP}; the elementary formulation is less notationally involved, and can serve as a useful means of conveying the basic ideas of the approach in a straightforward manner. In this paper we focus on the general formulation of the ECFP algorithm.

In ECFP, players are grouped into sets of equivalence classes, or ``permutation invariant'' classes. Such grouping allows players to analyze collective behavior by tracking only the statistics of each equivalence class, rather than tracking the statistics of every individual player.

Let $m\leq n$, denote the number of classes, let $I = \{1,\ldots,m\}$ be an index set, and let $\mathcal{C}=\{C_1,\ldots,C_m\}$ be a collection of subsets of $N$; i.e. $C_k \subseteq N, ~\forall k\in I$. A collection $\mathcal{C}$ is said to be a \textit{permutation-invariant partition} of $N$ if,\\
$(i)$ $C_k \cap C_\ell = \emptyset$, for $k,\ell \in I$, $k\not= \ell$, \\
$(ii)$ $\bigcup\limits_{k\in I} C_k = N$,\\
$(iii)$ for $k\in I$, $i,j\in C_k$, $Y_i = Y_j$,\\
$(iv)$ for $k\in I$, $i,j\in C_k$, there holds for any strategy profile $y = (y_i,y_j,y_{-(i,j)}) \in Y$,
$$ u(y_i,y_j,y_{-(i,j)}) = u([y_j]_i,[y_i]_j,y_{-(i,j)}),$$
where the notation $([y_i]_j,[y_j]_i,y_{-(i,j)})$ indicates a permutation of (only) the strategies of players $i$ and $j$ in the strategy profile $y =(y_i,y_j, y_{-(i,j)})$.

For a collection $\mathcal{C}$, define $\phi(\cdot):N\rightarrow I$ to be the unique mapping such that $\phi(i) = k$ if and only if $i\in C_k$.

For $k\in I$, and $p\in \Delta^n$, and permutation-invariant partition $\mathcal{C}$,  define
\begin{equation}
\label{def_partial_centroid}
\bar p^k := |C_k|^{-1}\sum_{i\in C_k} p_i
\end{equation}
to be the \textit{$k$-th centroid} with respect to $\mathcal{C}$, where $|C_k|$ denotes the cardinality of the set $C_k$. Likewise for $p\in \Delta^n$ define
\begin{equation}
\label{def_centroid}
\bar p := (\bar p_1, \bar p_2,\ldots, \bar p_n),
\end{equation}
where $\bar p_i := \bar p ^{\phi(i)}$, to be the \textit{centroid distribution} with respect to $\mathcal{C}$.

Given a permutation-invariant partition $\mathcal{C}$, let the set of symmetric Nash equilibria (relative to $\mathcal{C}$) be given by,
\begin{equation}
SNE := \{p\in NE: p_i = p_j~ \forall i,j \in C_k,~\forall~ k\in I\},
\end{equation}
and let the set of mean-centric equilibria (relative to $\mathcal{C}$) be given by,
\begin{equation}
MCE := \{p\in \Delta^n:~ U(p_i,\bar p_{-i}) \geq U(p_i',\bar p_{-i}), \forall p_i' \in \Delta_i,~\forall i\}.
\end{equation}
The set of MCE is neither a strict superset nor subset of the NE---rather, it is a set of natural equilibrium points tailored to the ECFP dynamics \cite{swenson2013MCE}. The set of SNE however, is contained in the set of MCE.

The sets of SNE and MCE relative to a partition $\mathcal{C}$ can be shown to be non-empty under A.\ref{a_iden_interests} using fixed point arguments similar to \cite{Cheng04,swenson2013MCE}.

\section{Empirical Centroid Fictitious Play}
\label{sec_DT_ECFP}
%Let the game $\Gamma$ be played repeatedly at times $t\in\{1,2,\ldots\}$. For $i\in N$, let $a_i(t)\in \Delta_i$ signify the action used by player $i$ in round $t$, and let $a(t) = (a_1(t),\ldots,a_n(t))$ signify the joint action.

Let the game $\Gamma$ be played repeatedly as in Section \ref{sec_rep_play}. Let the empirical distribution for player $i$ be formed recursively with $q_i(1) = a_i(1)$ and for $t\geq1$,
\begin{equation}
\label{def_qt}
q_i(t+1) = q_i(t) + \gamma_t\left( a_i(t+1) - q_i(t)\right),
\end{equation}
where we assume:
%$\{\gamma_t\}_{t\geq 1}$ is a sequence of non-negative numbers satisfying the assumption
\begin{assumption}
\label{a_step_size}
The sequence $\{\gamma_t\}_{t\geq 1}$ in \eqref{def_qt} satisfies $\gamma_t \geq 0,~\forall t$, $\sum_{t\geq 1} \gamma_t = \infty$, and $\lim_{t\rightarrow\infty} \gamma_t = 0$.
\end{assumption}
\noindent Let the joint empirical distribution be given by $q(t) := (q_1(t),\ldots,q_n(t))$.

Typical FP-type learning algorithms\footnote{We use the term FP-type learning algorithm to refer to an algorithm in which players choose their next-stage action as a myopic best response to some forecast rule based on the current time-averaged empirical distribution of play; cf. the learning framework considered in \cite{jordan1993}.} consider the empirical distribution to be a time-averaged histogram that places equal weight on all rounds; this corresponds to a step size of form $\gamma_t = \frac{1}{t+1},\forall t$ (e.g., \eqref{def_qt_FP}). If a FP-type algorithm retains its fundamental learning properties under the more general assumption A.\ref{a_step_size}, then we say the algorithm is \emph{step-size robust}.

%The work \cite{leslie2006generalised} first to considered this type of generalization, and showed that FP is in fact step-size robust. In this work we show that ECFP is also step-size robust.

%The work \cite{leslie2006generalised} demonstrated that the convergence properties of classical FP are in fact robust to alternate choices of step size sequences, so long as A.\ref{a_step_size} is satisfied. In this work, show that the convergence properties of ECFP are also robust to alternate step size sequences, so long as A.\ref{a_step_size} is satisfied (i.e., ECFP is step-size robust).

In ECFP \cite{swenson2012ECFP}, players do not track the empirical distribution of each individual player. Instead, they track only the centroid $\bar q^k(t)$ for each $k\in I$ (see \eqref{def_partial_centroid}).
%(This generalization can be useful in large games where it helps mitigate the amount of information that needs to be tracked by each player.)
Intuitively speaking, in ECFP each player $i$ assumes (perhaps incorrectly) that for each class $C_k\in\mathcal{C}$ the centroid $\bar q^k(t)$ accurately represents the mixed strategy for all players $j\in C_k$. Each player $i$ chooses her next-stage action as a myopic best response given this assumption.

Formally, the joint action at time $(t+1)$ is chosen according to the rule\footnote{The action $a(1)$ may be chosen arbitrarily.}
\begin{equation}
a(t+1) \in BR^{\epsilon_t}(\bar q(t)),
\label{dt_action_process}
\end{equation}
where $\bar q(t)$ is the centroid distribution associated with $q(t)$ (see \eqref{def_centroid}), and where it is assumed that,
\begin{assumption}
The sequence $\{\epsilon_t\}_{t\geq 1}$ in \eqref{dt_action_process} satisfies $\lim_{t\rightarrow\infty} \epsilon_t = 0$.
\label{a_epsilon}
\end{assumption}
Typical FP-type learning algorithms assume that players are always perfect optimizers; i.e., $\epsilon_t = 0,~\forall t$.
If a FP-type learning algorithm retains its fundamental learning properties under A.\ref{a_epsilon}, we say the algorithm is \emph{best-response robust}. The work \cite{leslie2006generalised} first considered generalizations to FP of the forms indicated in A.\ref{a_step_size}--A.\ref{a_epsilon} and showed that classical FP is both step-size robust and best-response robust. In this work we show that ECFP is also robust in both these senses.

Combining \eqref{def_qt} with \eqref{dt_action_process}, gives the following difference inclusion governing the behavior of $\{q(t)\}_{t\geq 1}$,
\begin{equation}
q(t+1) \in \left(1 - \gamma_t\right)q(t) + \gamma_t BR^{\epsilon_t}(\bar q(t)).
\label{dt_ecfp}
\end{equation}
Likewise, Lemma \ref{DT_centroid_lemma} (see appendix) shows that the sequence of centroid distributions $\{\bar q(t)\}_{t\geq 1}$ follows the difference inclusion,
\begin{equation}
\bar q(t+1) \in \left(1 - \gamma_t \right)\bar q(t) + \gamma_t BR^{\epsilon_t}(\bar q(t)).
\label{dt_centroid_process}
\end{equation}
We refer to the sequence $\{q(t),\bar q(t)\}_{t\geq 1}$ as a discrete-time ECFP (DT-ECFP) process with respect to $(\Gamma, \mathcal{C})$.

The following theorem is the main result of the paper---it states that, if $\Gamma$ is an identical interests game, then under the relatively weak assumptions A.\ref{a_step_size}--A.\ref{a_epsilon}, players engaged in ECFP asymptotically learn elements of sets of SNE and MCE. Learning of MCE occurs in the sense that $d(q(t),MCE) \rightarrow 0$---this form of learning corresponds to the typical notion of setwise \emph{convergence in empirical distribution} typical in classical FP (see Section \ref{sec_FP} and \cite{fudenberg1998theory},\cite{young2004strategic}). Learning of SNE occurs in the sense that $d(\bar q(t),SNE) \rightarrow 0$. This notion of learning, while similar in spirit to the typical notion of convergence in empirical distribution, differs in that it is the empirical centroid distribution \eqref{def_centroid} that is converging to the set of SNE, rather than the empirical distribution itself.

\begin{theorem}
Assume A.\ref{a_iden_interests}--A.\ref{a_epsilon} hold. Let $\mathcal{C}$ be a permutation-invariant partition of the player set $N$. Let $\{q(t),\bar q(t)\}_{t\geq 1}$ be an ECFP process with respect to $(\Gamma,\mathcal{C})$. Then,\\
(i) players learn a subset of the MCE in the sense that $\lim_{t\rightarrow\infty} d(q(t),MCE) = 0$,\\
(ii) players learn a subset of the SNE in the sense that $\lim_{t\rightarrow\infty} d(\bar q(t),SNE) = 0$.
\label{main_result}
\end{theorem}

We note that if $\epsilon_t = 0$ and $\gamma_t = \frac{1}{t+1}$, then convergence of ECFP in the sense of Theorem \ref{main_result} was established in our prior work \cite{swenson2012ECFP}.

In order to prove Theorem \ref{main_result} in its full generality we follow the approach of \cite{leslie2006generalised,benaim2005stochastic}---we first study the set of continuous-time differential inclusions associated with ECFP, and then derive Theorem \ref{main_result} from the continuous-time results via tools from the theory of stochastic approximations.
%The proof of Theorem \ref{main_result} is found in Section \ref{sec_proof_main_result}.

%This section relates the limit sets of DT-ECFP to the limit sets of CT-ECFP using the tools presented in Section \ref{sec_CTS}.
In particular, Section \ref{sec_CTS} discusses the notion of a perturbed solution of a differential inclusion, introduces the notion of a chain transitive set, and presents key results that allow one to relate the limit sets of perturbed solutions to internally chain transitive sets of the associated differential inclusion. Section \ref{sec_CT_ECFP} then presents continuous-time ECFP (CT-ECFP) and shows convergence of CT-ECFP to the sets of SNE and MCE using Lyapunov arguments.

Section \ref{sec_DT_convergence} presents Lemmas \ref{lemma_limit_set_of_ECFP} and \ref{lemma_CTS_in_SNE} that relate the limit sets of DT-ECFP to the limit sets of CT-ECFP. Lemma \ref{lemma_limit_set_of_ECFP} shows that the limit sets of DT-ECFP are contained in the internally chain transitive sets of the corresponding CT-ECFP process. This is accomplished by first showing that DT-ECFP processes may be considered to be perturbed solutions of the associated CT-ECFP differential inclusion, and then invoking Theorem \ref{limit_set_thrm} to clinch the result. Lemma \ref{lemma_CTS_in_SNE} then shows that the internally chain transitive sets of CT-ECFP are contained in the sets of MCE and SNE. This is accomplished by invoking Proposition \ref{lyapunov_prop} together with the Lyapunov arguments derived for CT-ECFP processes in Section \ref{sec_CT_ECFP}.

The proof of Theorem \ref{main_result} then follows by combining Lemmas \ref{lemma_limit_set_of_ECFP} and \ref{lemma_CTS_in_SNE}, as noted in Section \ref{sec_proof_main_result}.

%In particular, in Section ?? we discuss differential inclusions, and introduce the notion of a chain transient set. In Section ?? we introduce CT-ECFP, and demonstrate through Lyapunov arguments convergence of CT-ECFP to a limit set.

\section{Chain Transient Sets}
\label{sec_CTS}
We study the limiting behavior of DT-ECFP by first studying the behavior of a continuous-time version of ECFP, and then relating the limit sets of DT-ECFP to the limit sets of its continuous-time counterpart. In particular, we will relate the limit sets of DT-ECFP to the chain transitive sets of CT-ECFP. Following the approach of \cite{benaim2005stochastic},\cite{leslie2006generalised}, let $F$ denote a set-valued function mapping each point $\xi\in \mathbb{R}^m$ to a set $F(\xi) \in \mathbb{R}^m$. We assume:
\begin{assumption}
\label{a_F}
(i) $F$ is a closed set-valued map.\footnote{I.e., $\mbox{Graph}(F) := \{(\xi,\eta):\eta\in F(\xi) \} $ is a closed subset of $\mathbb{R}^m \times \mathbb{R}^m$.} \\
(ii) $F(\xi)$ is a nonempty compact convex subset of $\mathbb{R}^m$ for all $\xi\in \mathbb{R}^m$.\\
%(iii) There exists $c>0$ such that for all $\xi\in \mathbb{R}^m$, $\sup_{\eta\in F(\xi)}\|\eta\| \leq c(1+\|\xi\|),$ where $\|\cdot\|$ denotes any norm on $\mathbb{R}^m$.
(iii) For some norm $\|\cdot\|$ on $\mathbb{R}^m$, there exists $c>0$ such that for all $\xi\in \mathbb{R}^m$, $\sup_{\eta\in F(\xi)}\|\eta\| \leq c(1+\|\xi\|).$
\end{assumption}

\begin{definition}
A solution for the differential inclusion $\frac{dx}{dt} \in F(x)$
with initial point $\xi \in \mathbb{R}^m$ is an absolutely continuous mapping $x:\mathbb{R}\rightarrow\mathbb{R}^m$ such that $x(0) = \xi$ and $\frac{dx(t)}{dt} \in F(x(t))$ for almost every $t\in \mathbb{R}$.
\end{definition}

%\begin{definition}
%A discrete time process $\{x(t)\}\t \geq 1$, $x(t) \in \mathbb{R}^m$ is a solution for \eqref{def_perturbed_soln} if it verifies a recursion of the form
%\begin{equation}
%\label{def_perturbed_soln}
%x({t+1}) - x({t}) - \alpha_{t+1} M({n+1}) \in \alpha_{t+1}F(x(t)),
%\end{equation}
%where $\{\alpha_t\}_{t\geq 1}$ is a sequence of non-negative numbers such that
%$$ \sum_{t\geq 1} \alpha_t = \infty \quad \mbox{ and } \quad \lim_{t\rightarrow\infty} \alpha_t = 0,$$
%and $\{M(t)\}_{t\geq 1}$ is a sequence of deterministic or random perturbations.
%\end{definition}

\begin{definition}
Let $\|\cdot\|$ be a norm on $\mathbb{R}^m$, and let $F:\mathbb{R}^m\rightarrow \mathbb{R}^m$ be a set valued function satisfying A.\ref{a_F}.
%be a closed set-valued map such that $F(x)$ is a non-empty compact convex subset of $\mathbb{R}^m$ with $\sup\{\|z\|:z\in F(x)\}\leq c(1+\|x\|)$ for all $x$.
Consider the differential inclusion
\begin{equation}
\frac{dx}{dt} \in F(x).
\label{CTS_diff_incl}
\end{equation}

(a) Given a set $X\subset\mathbb{R}^m$ and points $\xi$ and $\eta$, we write $\xi\hookrightarrow \eta$ if for every $\epsilon >0$ and $T>0$ there exist an integer $n\geq 1$, solutions $x_1,\ldots,x_n$ to the differential inclusion \eqref{CTS_diff_incl}, and real numbers $t_1,\ldots,t_n$ greater than $T$ such that\\
(i) $x_i(s)\in X$, for all $0\leq s \leq t_i$ and for all $i=1,\ldots,n$,\\
(ii) $\|x_i(t_i) - x_{i+1}(0)\| \leq \epsilon$ for all $i=1,\ldots,n-1$,\\
(iii) $\|x_1(0) - \xi\| \leq \epsilon$ and $\|x_n(t_n) - \eta\|\leq \epsilon.$

(b) $X$ is said to be internally chain transient if $X$ is compact and $\xi \hookrightarrow \xi$ for all $\xi\in X$.
\end{definition}

The following theorem from \cite{benaim2005stochastic} allows one to relate the set of limit points of certain discrete-time processes to the internally chain transient sets of their continuous-time counterparts.

\begin{theorem}
Assume $F:\mathbb{R}^m\rightarrow \mathbb{R}^m$ is a set valued function satisfying A.\ref{a_F}.
%closed set-valued map such that $F(x)$ is a non-empty compact convex subset of $\mathbb{R}^m$ with $\sup\{\|z\|:z\in F(x)\}\leq c(1+\|x\|)$ for all $x$.
Let $\{x(t)\}_{t\geq 1}$ be a process satisfying
\begin{equation}
\label{def_perturbed_soln}
x({t+1}) - x({t}) - \alpha_{t+1} M({n+1}) \in \alpha_{t+1}F(x(t)),
\end{equation}
where $\{\alpha_t\}_{t\geq 1}$ is a sequence of non-negative numbers such that
$$ \sum_{t\geq 1} \alpha_t = \infty \quad \mbox{ and } \quad \lim_{t\rightarrow\infty} \alpha_t = 0,$$
and $\{M(t)\}_{t\geq 1}$ is a sequence of deterministic or random perturbations. If
\begin{align}
(a) & ~ \mbox{for all } T>0,\\
& \lim\limits_{t\rightarrow\infty}\sup\limits_{k} \left(\|\sum\limits_{i=t}^{k-1}\alpha_{i+1} M({i+1}) \|:\sum\limits_{i=t}^{k-1}\alpha_{i+1} \leq T \right) = 0,\\
(b) & ~ \sup_{t\geq1}\|x(t)\|<\infty,
\end{align}
then the set of limit points of $\{x(t)\}_{t\geq1}$ is a connected internally chain transitive set of the differential inclusion
$$\frac{d}{dt}x(t) \in F(x(t)).$$
\label{limit_set_thrm}
\end{theorem}
In an abuse of terminology,\footnote{A perturbed solution of \eqref{CTS_diff_incl} typically refers to a continuous time process that is associated with \eqref{CTS_diff_incl} by means of an integrable perturbation. Under appropriate conditions, processes of the form \eqref{def_perturbed_soln} may be transformed via an interpolation procedure into a continuous-time process satisfying the typical definition of a perturbed solution. See \cite{benaim2005stochastic}, Section I for more details.} we sometimes refer to a discrete-time process $\{x(t)\}_{t\geq 1}$ verifying the recursion \eqref{def_perturbed_soln} as a perturbed solution of \eqref{CTS_diff_incl}.

The differential inclusion \eqref{CTS_diff_incl} induces a set-valued dynamical system $\{\Phi_t\}_{t\in \mathbb{R}}$ defined by
$$\Phi_t(x_0) := \{ x(t):  x \mbox{ is a solution to \eqref{CTS_diff_incl} with } x(0) = x_0\}.
$$
Let $\Lambda$ be any subset of $\mathbb{R}^m$. A continuous function $V:\mathbb{R}^m\rightarrow\mathbb{R}$ is called a Lyapunov function for $\Lambda$ if $V(y) < V(x_0)$ for all $x_0\in \mathbb{R}^m\backslash \Lambda$, $y\in \Phi_t(x_0),~t>0$, and $V(y) \leq V(x_0)$ for all $x_0\in \Lambda,~ y\in \Phi_t(x_0)$ and $t\geq 0$. The following proposition (\cite{benaim2005stochastic}, Proposition 3.27) allows one to relate the chain transitive sets of a differential inclusion to Lyapunov attracting sets.
\begin{prop}
\label{lyapunov_prop}
Suppose that $V$ is a Lyapunov function for $\Lambda$. Assume that $V(\Lambda)$, the image of $\Lambda$ under $V$, has empty interior. Then every internally chain transitive set $L$ is contained in $\Lambda$ and $V\vert L$, the restriction of $V$ to the set $L$, is constant.
\end{prop}
%\begin{cor}
%\label{lyapunov_cor}
%Suppose that $V$ is a Lyapunov function for $\Lambda$ and suppose that $V$ is $C^m$ for $m\geq 1$ and $\Lambda$ is contained in the critical points set of $V$. Then every internally chain transitive set $L$ lies in $\Lambda$ and $V\vert \L$ is constant.
%\end{cor}

\section{continuous-time ECFP}
\label{sec_CT_ECFP}
In this section we consider a continuous-time version of ECFP. Let $\Gamma$ satisfy A.\ref{a_iden_interests} and let $\mathcal{C}$ be a permutation-invariant partition of $N$. In analogy to\footnote{Note that \eqref{dt_ecfp} may be written as $q(t+1)-q(t) \in \frac{1}{t+1}\left(BR^{\epsilon_t}(\bar q(t)) - q(t) \right)$.} \eqref{dt_ecfp}, for $t> 0$ let
\begin{equation}
\label{ct_ecfp}
\dot q^c(t) \in BR(\bar q^c(t)) - q^c(t),
\end{equation}
where we use the superscript $q^c(t)$ to indicate a continuous-time analog of the empirical distribution, and where, for $p\in \Delta^n$, we let $BR(p) := BR^\epsilon(p)$ with $\epsilon=0$, and $\bar q^c(t)$ is the centroid distribution associated with $q^c(t)$ (see \eqref{def_centroid}). We refer to the process $\{q^c(t),\bar q^c(t)\}_{t\geq 0}$ as a continuous-time ECFP (CT-ECFP) process relative to $(\Gamma,\mathcal{C})$.
%\begin{remark}
%Note that the constant solutions of \eqref{ct_ecfp} are precisely the mean-centric equilibria of \eqref{dt_ecfp}.
%\label{remark_MCE}
%\end{remark}

As our end goal involves studying the limiting behavior of $\{\bar q^c(t)\}_{t\geq 1}$, note that for $k\in I$, and $\bar{q}^{c,k}(t)$ defined similar to \eqref{def_partial_centroid}, there holds
\begin{align}
\dot{\bar{q}}^{c,k}(t) & = \frac{d}{dt}|C_k|^{-1}\sum\limits_{j\in C_k} q^c_j(t)\\
& = |C_k|^{-1}\sum\limits_{j\in C_k} \frac{d}{dt} q^c_j(t)\\
& =|C_k|^{-1}\sum\limits_{j\in C_k} \dot q^c_j(t),
\end{align}

Let $p(t) = \dot q^c(t) + q^c(t)$, so that $\bar p(t) = (\bar p_1(t),\ldots,\bar p_n(t))$ with $\bar p_i(t) = \bar p^{\phi(i)}(t)$ for $i\in N$ (see \eqref{def_centroid}). By the above, and the linearity of differentiation, $\bar p(t) = \dot{\bar{q^c}}(t) + \bar q^c(t)$. Thus, by Lemma \ref{lemma_permutation_BR}, \eqref{ct_ecfp} implies that $\bar p(t) \in BR(\bar q^c(t)),$ or equivalently,
\begin{equation}
\dot{\bar{q^c}}(t) \in BR(\bar q^c(t)) - \bar q^c(t).
\label{centroid_diff_incl}
\end{equation}
%\begin{remark}
%Note that the constant solutions of \eqref{centroid_diff_incl} are precisely the symmetric Nash equilibria of $\Gamma$.
%\label{remark_SNE}
%\end{remark}

\subsection{Convergence in Continuous Time}
\label{sec_ct_results}
This section studies the convergence of continuous-time ECFP to the sets of SNE and MCE.

For any solution $q^c(t)$  of \eqref{ct_ecfp} and associated centroid process $\bar q^c(t)$, let $w(t) := U(\bar q^c(t))$ and let $v(t) :=
\frac{1}{n}\sum_{i=1}^n U(q^c_i(t),\bar q^c_{-i}(t))$. There holds,
\begin{align}
\dot w(t) & = \sum\limits_{i=1}^n \frac{\partial}{\partial \bar q^c_i}U(\bar q^c(t))\dot{\bar{q^c}}_i(t)\\
& \geq \sum\limits_{i=1}^n \left[U(\dot{\bar{q^c}}_i(t) + \bar q^c_i(t),\bar q^c_{-i}(t)) - U(\bar q^c(t))\right]\\
& = \sum\limits_{i=1}^n \left[\max_{\alpha_i \in \Delta_i} U(\alpha_i,\bar q^c_{-i}(t)) - U(\bar q^c(t)) \right]\geq 0,
\label{lyapunov_eq}
\end{align}
where the second line follows from the concavity of $U$ in $p_i$, and the third follows from \eqref{centroid_diff_incl}.

By  Lemma \ref{lemma_rearangement} there holds
\begin{equation}
\frac{1}{n}\sum_{i=1}^n U(q^c_i(t),\bar q^c_{-i}(t)) = U(\bar q^c(t)).
\label{rearrangement_identity}
\end{equation}
Hence $v(t) = w(t)$, there holds $\cdot v(t) \geq 0$. Moreover, the following expansion is useful in order to study $v$ as a Lyapunov function for the set of MCE:
{\small
\begin{align}
\dot v(t) = \dot w(t) & \geq \sum\limits_{i=1}^n \left[\max_{\alpha_i \in \Delta_i} U(\alpha_i,\bar q^c_{-i}(t)) - U(\bar q^c(t))\right]\\
& = \sum\limits_{i=1}^n \max_{\alpha_i \in \Delta_i}U(\alpha_i,\bar q^c_{-i}(t)) - nU(\bar q^c(t))\\
& = \sum\limits_{i=1}^n \max_{\alpha_i \in \Delta_i}U(\alpha_i,\bar q^c_{-i}(t)) - \sum_{i=1}^n U(q^c_i(t),\bar q^c_{-i}(t))\\
& = \sum\limits_{i=1}^n \left[ \max_{\alpha_i \in \Delta_i}U(\alpha_i,\bar q^c_{-i}(t)) - U(q^c_i(t),\bar q^c_{-i}(t))\right] \geq 0,
\label{lyapunov_eq2}
\end{align}}
where the inequality follows from \eqref{lyapunov_eq}, and the third line follows again from \eqref{rearrangement_identity}.

%Furthermore, by \eqref{ct_ecfp} and \eqref{centroid_diff_incl}, $U(\dot{\bar{q}}_i(t) + \bar q_i(t),\bar q_{-i}(t)) = U(\dot{{q}}_i(t) + q_i(t),\bar q_{-i}(t)), ~\forall i$, and in particular
%$$\sum\limits_{i=1}^n U(\dot{\bar{q}}_i(t) + \bar q_i(t),\bar q_{-i}(t)) - U(\bar q(t)) = \sum\limits_{i=1}^n U(\dot{{q}}_i(t) + q_i(t),\bar q_{-i}(t)) - U(\bar q(t)).$$
%Thus by \eqref{lyapunov_eq},
%\begin{align}
%\dot v(t) = \dot w(t) & \geq \sum\limits_{i=1}^n U(\dot{\bar{q}}_i(t) + \bar q_i(t),\bar q_{-i}(t)) - U(\bar q(t))\\
%& = \sum\limits_{i=1}^n U(\dot{{q}}_i(t) + q_i(t),\bar q_{-i}(t)) - U(\bar q(t)) \geq 0.
%\label{lyapunov_eq2}
%\end{align}

By \eqref{lyapunov_eq}, $w(t)$ is weakly increasing, and is constant in a time interval $T$ if and only if $\max_{\alpha_i \in \Delta_i} U(\alpha_i,\bar q^c_{-i}(t)) - U(\bar q^c(t)) = 0, ~\forall i$; i.e., if and only if $\bar q^c(t) \in SNE$ for all $t\in T$.

By \eqref{lyapunov_eq2}, $v(t)$ is weakly increasing, and $\dot v(t) = 0$ in some interval $T$  implies $\max_{\alpha_i \in \Delta_i} U(\alpha_i,\bar q^c_{-i}(t)) - U(q^c_i(t),\bar q^c_{-i}(t)) =0, ~\forall i\in N$, $t\in T$; i.e., $q^c(t) \in MCE$ for all $t\in T$. Moreover, by Lemma \ref{lemma_permutation_BR}, $q^c(t) \in MCE, ~\forall t\in T \implies \bar q^c(t) \in SNE \forall t\in T $, which by the above comments implies $\dot w(t) = 0$ in $T$, or equivalently $\dot v(t) = 0$ in $T$. Thus, $v(t)$ is constant in a time interval $T$ if and only if $q^c(t)\in MCE$ for all $t\in T$.

\begin{prop}
Assume A.\ref{a_iden_interests} holds. Then,\\
(i) The limit set of every solution of \eqref{centroid_diff_incl} is a connected subset of SNE along which $U$ is constant;\\
(ii) For $p\in \Delta^n$, let
$V(p) := \frac{1}{n}\sum_{i=1}^n U(p_i,\bar q^c_{-i}).$
The limit set of every solution of \eqref{ct_ecfp} is a connected subset of MCE along which $V$ is constant.\\
\end{prop}

\noindent The proof of this proposition follows from the above comments.

\section{Limit Sets of Discrete-Time ECFP}
\label{sec_DT_convergence}
In this section we study the limit sets of DT-ECFP by relating them to the internally chain transitive sets of CT-ECFP.
%This section relates the limit sets of DT-ECFP to the limit sets of CT-ECFP using the tools presented in Section \ref{sec_CTS}. Lemma \ref{lemma_limit_set_of_ECFP} shows that the limit sets of DT-ECFP are contained in the internally chain transient sets of the corresponding CT-ECFP process. This is accomplished by showing that DT-ECFP processes may be considered to be a bounded perturbed solution of the associated CT-ECFP differential inclusion and then invoking Theorem \ref{limit_set_thrm}. (note: define bps.)
%
%Lemma \ref{lemma_CTS_in_SNE} then shows that the internally chain transitive sets of CT-ECFP are contained in the sets of MCE and SNE. This is accomplished by invoking Proposition \ref{lyapunov_prop} together with the Lyapunov arguments derived for CT-ECFP processes in Section \ref{sec_CT_ECFP}.

The following lemma relates the limit sets of DT-ECFP to the internally chain transient sets of the CT-ECFP differential inclusions \eqref{ct_ecfp} and \eqref{centroid_diff_incl}.
\begin{lemma}
Assume A.\ref{a_iden_interests}--A.\ref{a_epsilon} hold. Let $\{q(t),\bar q(t)\}_{t\geq 1}$ be a discrete-time ECFP process. Then, \\
(i) The set of limit points of $\{\bar q(t)\}_{t\geq 1}$ is a connected internally chain transient set of \eqref{centroid_diff_incl},\\
(ii) The set of limit points of $\{ q(t)\}_{t\geq 1}$ is a connected internally chain transient set of \eqref{ct_ecfp}.
\label{lemma_limit_set_of_ECFP}
\end{lemma}
\begin{proof}
Proof of (ii):
Observe that adding and subtracting the set $BR(\bar q(t))$ and rearranging terms in \eqref{dt_ecfp} gives,
\begin{align}
q(t+1) - q(t) + \gamma_t\left\{BR(\bar q(t)) - BR^{\epsilon_t}(\bar q(t))\right\}\\
\in \gamma_t \left\{BR(\bar q(t)) - q(t)\right\}.
\label{lemma_limit_set_of_ECFP_eq1}
\end{align}
Thus, the process $\{q(t)\}_{t\geq 1}$ fits the template of Theorem \ref{limit_set_thrm} with $x(t) := q(t)$, $F(x) := BR(\bar x) - x$, and $M(t) := BR(\bar x) - BR^{\epsilon_t}(\bar x)$. It is straightforward to verify that $F$ satisfies A.\ref{a_F}.
% meets the requirements of Theorem \ref{limit_set_thrm}: $F$ is a closed set-valued map such that $F(x)$ is a non-empty compact convex subset of $\mathbb{R}^m$ with $\sup\{\|z\|:z\in F(x) \}\leq c(1+\|x\|)$ for all $x$.

It suffices to show that the process \eqref{lemma_limit_set_of_ECFP_eq1} satisfies conditions (a) and (b) of Theorem \ref{limit_set_thrm}. Condition (b) is trivially satisfied since $q(t) \in \Delta^n$ for all $t$.

If $k$ is such that $\sum_{i=n}^{k-1} \gamma_i \leq T$, then
\begin{align}
\sup\limits_{k} \|\sum\limits_{i=n}^{k-1}\gamma_i \{BR^{\epsilon_i}(\bar q(t)) - BR(\bar q(t))\}\|\\
\leq T\sup\limits_{k} \|BR^{\epsilon_i}(\bar q(t)) - BR(\bar q(t))\|.
\end{align}
Since $BR$ is upper semicontinuous, $BR^{\epsilon}(p) \rightarrow BR(p)$ uniformly as $\epsilon \rightarrow 0$. Thus condition (a) holds, and (ii) is proved.

%If $BR^{\epsilon}(p) \rightarrow BR(p)$ uniformly as $\epsilon \rightarrow 0$ then the result follows. However, this is immediate from the upper semi-continuity of $BR(\cdot)$. Thus condition (a) holds, and (i) is proved.

Proof of (i): Adding and subtracting $BR(\bar q(t))$ and rearranging terms, the difference inclusion \eqref{dt_centroid_process} may be written as
\begin{align}
\bar q(t+1) - \bar q(t) + \gamma_t \left\{BR(\bar q(t))-BR^{\epsilon_t}(\bar q(t))\right\}\\
\in \gamma_t \left\{BR(\bar q(t)) - \bar q(t)\right\}.
\label{lemma_limit_set_of_ECFP_eq2}
\end{align}
Thus, the process $\{\bar q(t)\}$ fits the template of Theorem \ref{limit_set_thrm} with $x(t) := \bar q(t)$, $M(t) := BR(\bar q(t))-BR^{\epsilon_t}(\bar q(t))$, and $F(x) := BR(x) - x$.
It was shown in \cite{benaim2005stochastic} that $F$ satisfies A.\ref{a_F}.
%the requirements of Theorem \ref{limit_set_thrm}: $F$ is a closed set-valued map such that $F(x)$ is a non-empty compact convex subset of $\mathbb{R}^m$ with $\sup\{\|z\|:z\in F(x) \}\leq c(1+\|x\|)$ for all $x$.
It is sufficient to show that the process \eqref{lemma_limit_set_of_ECFP_eq2} satisfies conditions (a) and (b) of Theorem \ref{limit_set_thrm}. Condition (b) is trivially satisfied since $\bar q(t) \in \Delta^n$ for all $t$. Since $M(t)$ is defined the same manner as in case (i), the proof that condition (a) is satisfied follows the same reasoning as in the proof of assertion (ii).
\end{proof}

The following lemma relates internally chain transitive sets of \eqref{ct_ecfp} to the set of MCE, and the internally chain transitive sets of \eqref{centroid_diff_incl} to the set of SNE; combined with Lemma \ref{lemma_limit_set_of_ECFP} this will prove Theorem \ref{main_result}.
\begin{lemma}
Let $\Gamma$ be an identical interests game. Let $\mathcal{C}$ be a permutation-invariant partition on $\Gamma$. Then,\\
(i) Every internally chain transitive set of \eqref{centroid_diff_incl} is contained in the set of SNE.\\
(ii) Every internally chain transitive set of \eqref{ct_ecfp} is contained in the set of MCE.
\label{lemma_CTS_in_SNE}
\end{lemma}
\begin{proof}
Proof of (i):
Let $W := -U$. By Section \ref{sec_ct_results} (in particular, see \eqref{lyapunov_eq}), $W$ is a Lyapunov function for the set of SNE with $x(t) := \bar q(t)$. Note that $W$ is multilinear and hence continuously differentiable.

For a differentiable function $f:\mathbb{R}^m\rightarrow\mathbb{R}$, we say $x\in\mathbb{R}^m$ is a critical point of $f$ if for $i=1,\ldots, m$, the partial derivative at $x$ is zero, i.e, $\frac{\partial}{\partial x_j} f(x) = 0$. By Sard's Theorem (\cite{hirsch1976differential}, p. 69), if $CP$ is the critical points set of $W$, then $W(CP)$ contains no intervals. By definition, the set of $NE$ is contained in the critical points set of $U$, and hence also contained in the critical points of $W$. Furthermore, by definition, $SNE \subset NE$, and hence the set $SNE$ is contained in the critical points set of $W$. Thus, by Proposition \ref{lyapunov_prop}, every internally chain transitive set of \eqref{centroid_diff_incl} is contained in the set $SNE$.\\
Proof of (ii):
Note that, by Lemma \ref{lemma_permutation_BR}, $p\in BR(\bar p) \implies \bar p \in BR(\bar p). $
%$$U(p_i,\bar p_{-i}) \geq U(p_i',\bar p_{-i}), \forall p_i \in \Delta_i \implies U(\bar p_i,\bar p_{-i}) \geq  U(p_i',\bar p_{-i}), \forall p_i \in \Delta_i.$$
Thus, $p\in MCE$ implies that $\bar p \in SNE$. Let $V:\Delta^n \rightarrow \mathbb{R}$, with $V(p) := \frac{1}{n}\sum_{i=1}^n U(p_i,\bar p_{-i})$, and note that by Lemma \ref{lemma_rearangement}, $V(p) = U(\bar p)$. Invoking again Sard's Theorem, $U(NE)$ contains no intervals, and hence $U(SNE) \subset U(NE)$ contains no intervals. Since $U(SNE)$ contains no intervals, $V(MCE)$ also contains no intervals.

By Section \ref{sec_CT_ECFP} (in particular, see \eqref{lyapunov_eq2}) the function $V$ is a Lyapunov function for the set of MCE with $x(t) := q(t)$. It follows from Proposition \ref{lyapunov_prop} that every chain transitive set of \eqref{ct_ecfp} is contained in $MCE$.
\end{proof}

\subsection{Proof of Theorem 1}
\label{sec_proof_main_result}
Theorem \ref{main_result} follows directly from Lemmas \ref{lemma_limit_set_of_ECFP} and \ref{lemma_CTS_in_SNE}.
\section{Conclusions}
\label{sec_conclusion}
Classical Fictitious Play (FP) is robust to alterations in the empirical distribution step-size process and robust to best-response perturbations. These robustness properties allow for interesting modifications to FP which can be of great practical value. Empirical Centroid Fictitious Play (ECFP) is a generalization of FP designed for large games. The paper showed that ECFP is also robust to step-size alterations and best-response perturbations. This result enables future research to consider practical modifications to ECFP, similar to those already developed for FP.

\section*{Appendix}
\begin{lemma}
\label{lemma_rearangement}
Let $\mathcal{C}$ be a partition of $N$, and for $p\in \Delta^n$ let $\bar p$ be as defined in \eqref{def_centroid}. Then $\frac{1}{n}\sum_{i=1}^n U(p_i,\bar p_{-i}) = U(\bar p).$
\end{lemma}
\begin{proof}
Let $I$ be an index set for $\mathcal{C}$ and let $m$ be the cardinality of $I$. For $k\in I$, and $j\in C_k$ note that
\begin{align}
|C_k| U(\bar p) & = |C_k|U(\bar p_j,\bar p_{-j}) =|C_k| U(\left[ |C_k|^{-1}\sum_{i\in C_k} p_i\right]_j,\bar p_{-j})\\
& =  \sum_{i\in C_k} U([p_i]_j,\bar p_{-j}) = \sum_{i\in C_k} U(p_i,\bar p_{-i}),
\end{align}
where the second line follow from the definition of $\bar p_j$ (see \eqref{def_centroid}), the third by multilinearity of $U$, and the fourth by permutation invariance of elements in $C_k$. Thus,
\begin{align}
& \frac{1}{n}\sum_{i=1}^n U(p_i,\bar p_{-i}) = \frac{1}{n}\sum_{k\in I} \sum_{i\in C_k} U(p_i,\bar p_{-i})\\
& = \frac{1}{n} \sum_{k\in I} |C_k| U(\bar p) = \frac{1}{n} \sum_{i=1}^n U(\bar p) = U(\bar p).
\end{align}
\end{proof}

\begin{lemma}
\label{lemma_permutation_BR}
Let $q\in \Delta^n$, let $\bar q$ be as defined in \eqref{def_centroid}, and let $\epsilon \geq 0$.  If $p\in BR^{\epsilon}(\bar q)$, then $\bar p \in BR^\epsilon(\bar q)$.
\end{lemma}
\begin{proof}
Let $i\in N$. Recall that $\bar p := (p_1,\ldots,p_n)$ with $p_i = p^{\phi(i)}$. There holds
\begin{align}
U(\bar p_i, \bar q_{-i}) & = U(\left[|C_{\phi(i)}|^{-1}\sum\limits_{j\in C_{\phi(i)}}p_j \right]_i,\bar q_{-i})\\
& = |C_{\phi(i)}|^{-1}\sum\limits_{j\in C_{\phi(i)}}U([p_j]_i,\bar q_{-i}(t))\\
& = |C_{\phi(i)}|^{-1}\sum\limits_{j\in C_{\phi(i)}}U([p_j]_i, [\bar q^{\phi(i)}]_j,\bar q_{-(i,j)})\\
& = |C_{\phi(i)}|^{-1}\sum\limits_{j\in C_{\phi(i)}}U([p_j]_j, [\bar q^{\phi(i)}]_i,\bar q_{-(i,j)})\\
& = |C_{\phi(i)}|^{-1}\sum\limits_{j\in C_{\phi(i)}}U([p_j]_j,\bar q_{-j}(t))\\
& \geq |C_{\phi(i)}|^{-1}\sum\limits_{j\in C_{\phi(i)}}\max_{p_j' \in \Delta_j} \left( U(\alpha_j,\bar q_{-j})- \epsilon\right)\\
& = |C_{\phi(i)}|^{-1}\sum\limits_{j\in C_{\phi(i)}}\max_{p_i' \in \Delta_i} \left( U(\alpha_i,\bar q_{-i}) - \epsilon \right)\\
& = \max_{\alpha_i \in \Delta_i} U(\alpha_i,\bar q_{-i}) - \epsilon,
\label{perm_inv_arg}
\end{align}
where the first line follows by the definition of $\bar p_i$ (see \eqref{def_centroid}), the second from the multilinearity of $U$, the fourth by permutation invariance of elements of $C_{\phi(i)}$, the sixth by the fact that, by hypothesis, $p_j \in BR_j^{\epsilon}(\bar q_{-j})$, and the seventh by permutation invariance of elements of $C_{\phi(i)}$. Since this holds for all $i \in N$, it follows that $\bar p \in BR^\epsilon(\bar q)$.
\end{proof}

\begin{lemma}
\label{DT_centroid_lemma}
Assume A.\ref{a_iden_interests} holds and suppose the action sequence $\{a(t)\}_{t\geq 1}$ is chosen according to \eqref{dt_action_process}. Then centroid process $\{\bar q(t)\}_{t\geq 1}$ follows the differential inclusion \eqref{dt_centroid_process}.
\end{lemma}
\begin{proof}
Note that $\bar q(t+1)$ may be written recursively as
$$\bar q(t+1) = \bar q(t) + \frac{1}{t+1}\left(\bar a(t+1) - \bar q(t)\right).$$
By Lemma \ref{lemma_permutation_BR}, $a(t+1) \in BR^{\epsilon_t}(\bar q(t))$ implies $\bar a(t+1) \in BR^{\epsilon_t}(\bar q(t))$. Substituting this into the above recursion and rearranging terms shows that $\{\bar q(t)\}$ follows the difference inclusion \eqref{dt_centroid_process}.
\end{proof}

%\begin{lemma}
%\label{lemma_F_properties}
%Let $F(x) : BR(\bar x) - x$, $F:\Delta^n\rightarrow\Delta^n$. Then $F$ is a closed set-valued map such that $F(x)$ is a non-empty compact convex subset of $\mathbb{R}^m$ with $\sup\{\|z\|:z\in F(x) \}\leq c(1+\|x\|)$ for all $x$.
%\end{lemma}
%\begin{proof}
%\end{proof}

\bibliographystyle{IEEEtran}
\bibliography{myRefs}

\end{document}